\documentclass[12pt]{article}
\usepackage {amsfonts,amsmath,amssymb}

\newcommand{\g}{\mathfrak g}
\newcommand{\q}{\mathfrak q}
\newcommand{\uf}{\mathfrak u}
\newcommand{\Of}{\mathcal O}
\newcommand{\Pf}{\mathcal P}
\newcommand{\Sf}{\mathcal S}
\newcommand{\X}{\mathcal X}
\newcommand{\ad}{\mathop{\mathrm{ad}}\nolimits}
\newcommand{\Ad}{\mathop{\mathrm{Ad}}\nolimits}

\newcommand{\GL}{\mathop{\mathrm{GL}}\nolimits}
\newcommand{\Ind}{\mathop{\mathrm{Ind}}\nolimits}
\newcommand{\Int}{\mathop{\mathrm{Int}}\nolimits}

\newcommand{\Pic}{\mathop{\mathrm{Pic}}\nolimits}
\newcommand{\tr}{\mathop{\mathrm{tr}}\nolimits}
\newcommand{\codim}{\mathop{\mathrm{codim}}\nolimits}
\newcommand{\id}{\mathrm{id}}
\newcommand{\gen}{^\mathrm{gen}}
\newcommand{\reg}{^\mathrm{reg}}
\newcommand{\infl}{^\mathrm{inf}}
\newcommand{\can}{^\mathrm{can}}

\usepackage{amsthm}
\newtheorem{theorem}{Theorem}
\newtheorem{prop}{Proposition}
\newtheorem{lemma}{Lemma}
\newtheorem*{conj}{Conjecture}
\newtheorem{cor}{Corollary}

\begin{document}

\title{Induced conjugacy classes, prehomogeneous varieties, and canonical parabolic subgroups}
\author{Werner Hoffmann}
\maketitle
\begin{abstract}
We prove the properties of induced conjugacy classes, without using the original proof by Lusztig and Spaltenstein in the unipotent case, by adapting Borho's simpler arguments for induced adjoint orbits. We study properties of equivariant fibrations of prehomogeneous affine spaces, especially the existence of relative invariants. We also detect prehomogeneous affine spaces as subquotients of canonical parabolic subgroups attached to elements of reductive groups in the sense of Jacobson-Morozov. These results are prerequisites for making the geometric expansion of the Arthur-Selberg trace formula more explicit.

\end{abstract}

This paper contains prerequisites for an explicit description of certain coefficients that appear on the geometric side of Arthur's trace formula for a reductive group~$G$ (see section~19 of~\cite{Ar2} for an introduction). Those coefficients grew out of an invariance argument that did not allow their determination. In previous work on groups of low rank (\cite{Fl}, \cite{Ho}), the coefficients in question had been related to certain prehomogeneous zeta integrals. In a forthcoming paper, we aim to generalise that approach to groups of general rank.

Before entering the analytical argument, one has to identify, in the structure of the group~$G$, the prehomogeneous vector spaces supporting those zeta integrals. We will see that they are closely related to the canonical parabolic subgroups $Q$ of the elements of~$G$. Actually, in the case of non-unipotent elements, prehomogeneous affine spaces appear as well. More general prehomogeneous varieties enter the stage in still another role, which can be explained as follows.

The final form of the geometric side of the trace formula is a sum indexed by conjugacy classes. At the outset, it is an integral of a sum indexed by cosets in parabolic subgroups $P$ with respect to their unipotent radicals~$N$. When combining contributions from various parabolic subgroups~$P$ to a conjugacy class of~$G$, the notion of an induced conjugacy class fits in naturally. This notion was introduced by Lusztig and Spaltenstein~\cite{LS} in the case of unipotent conjugacy classes, and we extend the required results to the case of arbitrary conjugacy classes. It is useful to view the process of induction as an inflation from a Levi component $M$ to $P$ followed by the induction proper from $P$ to~$G$. The definition of inflation can be phrased in terms of prehomogeneous varieties.

The application of these ideas to the trace formula is hampered by the fact that the elements of an inflated conjugacy class in $P$ have different canonical parabolics $Q$ in general. This necessitates the application of a mean-value formula, for which we have to make sure that certain prehomogeneous affine spaces are special. I was not able to prove this in general and leave it as a conjecture.

The arguments in this paper are purely algebraic and independent of the trace formula. Since they may be of wider interest in the theory of linear algebraic groups, I decided to single them out in a separate paper.

Here are some notational conventions. If a group~$G$ acts on a set~$V$, the subset of fixed points of an element $\gamma$ of $G$ will be denoted by~$V^\gamma$ and the stabiliser of an element $\xi$ of $V$ by~$G^\xi$. The smallest $G$-invariant subset of~$V$ containing a given subset~$S$ will be written as~$[S]_G$. If $V$ is a normal subgroup of~$G$, on which $G$ acts by inner automorphisms, then $V^\gamma$ becomes the centraliser of~$\gamma$. Its trivial connected component will be denoted by~$V_\gamma$. If $F$ is a field and if an algebraic $F$-group $G$ acts $F$-morphically on an $F$-variety~$V$, then a geometric $F$-orbit is a minimal $G$-invariant $F$-closed subset of~$V$. The Lie algebra of a linear algebraic group will be denoted by the corresponding lower-case gothic letter. Unless stated otherwise, algebraic varieties $V$ are defined over an algebraically closed field~$E$, in which case we identify $V$ with the set~$V(E)$. From section~2 onwards, we have to assume that the $E$ has characteristic zero.

This work was supported by the SFB~701 of the German Research Foundation.

\section{Induction of conjugacy classes}

The notion of induced unipotent conjugacy classes was introduced in~\cite{LS} as a generalisation of the notion of a Richardson class attached to a parabolic subgroup. There is a parallel notion of induced nilpotent adjoint orbits. The induction of general adjoint orbits has been present for a while (cf. \cite{Bo},~\cite{Ke}), whereas the corresponding notion at the group level has been studied only recently (cf.~\cite{CE}), although it was mentioned in~\cite{Ar}, p.~255. The method in~\cite{CE} (and in a previous version of this paper) was reduction to the unipotent case. Since the original proofs in~\cite{LS} for that case are quite involved, we will rather transfer the easy direct arguments from~\cite{Bo} to the group case.

\begin{prop}\label{defind}
Let $G$ be a connected reductive algebraic group. Let $P$ be a parabolic subgroup of $G$ with unipotent radical $N$ and some Levi component~$M$.
\begin{enumerate}
\item[(i)] Given a conjugacy class $C$ in~$M$, there is a unique conjugacy class $\tilde C$ in $G$, called the class induced from $C$ via~$P$, such that $\tilde C\cap CN$ is open and dense in~$CN$.
\item[(ii)] If, in the situation of~(i), we have $N_\mu=\{1\}$ for some (hence any) $\mu\in C$, then $CN$ is a conjugacy class of~$P$, and $\tilde C=[C]_G$.
\item[(iii)] If $D$ is a conjugacy class of~$G$, then there can be only finitely many conjugacy classes of~$M$ whose induction via~$P$ yields~$D$.
\item[(iv)] If $G$, $P$, $M$ and $C$ are defined over a subfield $F$ of~$E$ and $C(F)\ne\emptyset$, then $\tilde C$ is defined over~$F$ and $\tilde C(F)\ne\emptyset$.
\end{enumerate}
\end{prop}
Actually, there is no need to fix a particular Levi component~$M$. We could identify $M$ with $P/N$ and $CN$ with the preimage of $C$ under the natural epimorphism $P\to P/N$. The varieties appearing in~(iv) can be given independently of~$E$ by their affine $F$-algebras, in which case $C$ and $\tilde C$ are geometric $F$-conjugacy classes, i.~e.\ geometric $F$-orbits for the action of the respective groups by inner automorphisms.

In the proof of this and other results we will use the following well-known lemma.
\begin{lemma}\label{gN}
Let $P$ be a linear algebraic group, $N$ a connected unipotent normal subgroup and $\delta\in P$ with semisimple component~$\sigma$. Then $\delta N=[\delta N_\sigma]_N$.
\end{lemma}
This follows from the proof of Lemma~2.1 of~\cite{Ar1}.
\begin{proof}[Proof of Proposition~\ref{defind}.]
If $D$ is a conjugacy class in~$G$, then the semisimple components of the elements of $D$ make up a conjugacy class~$D_{\mathrm{s}}$. By~\cite{Ri} and~\cite{Lu}, there are only finitely many unipotent conjugacy classes in $G_\sigma $ for a given semisimple element~$\sigma$, hence there are only finitely many conjugacy classes $D$ in $G$ with prescribed semisimple class~$D_{\mathrm{s}}$.

Now let $\mu\in C$ with semisimple component~$\sigma$. It follows from Lemma~\ref{gN} that the semisimple component of every element of $\mu N$ is $N$-conjugate to~$\sigma$. Thus, the semisimple components of every element of~$CN=[\mu N]_P$ is $P$-conjugate to~$\sigma$. By the preceding remarks, $CN$ meets only finitely many $G$-conjugacy classes, so there is a conjugacy class $\tilde C$ in $G$ such that the set $\tilde C\cap CN$ is dense in~$CN$. The former set must be open in the latter because $\tilde C$ is open in its closure. Since $CN$ is irreducible, $\tilde C$ unique. This proves~(i) and shows that $C_{\mathrm{s}}\subset\tilde C_{\mathrm{s}}$.

In the case $N_\mu=\{1\}$, Lemma~\ref{gN} shows that $\mu N=[\mu]_N$, hence $CN=[\mu N]_M=[\mu]_P$. Therefore, $\mu\in\tilde C$, and (ii) follows. 

By results of~\cite{Ri}, $D_{\mathrm{s}}\cap M$ consists of finitely many $M$-conjugacy classes, hence there are only finitely many classes $C$ in $M$ such that~$C_{\mathrm{s}}\subset D_{\mathrm{s}}$, and (iii) is proved.

If a coset $\mu N$ with $\mu\in C$ did not meet~$\tilde C$, neither would $[\mu N]_P=CN$, which contradicts our construction. Thus, the open subset $\tilde C\cap\mu N$ of $\mu N$ is non-empty. The connected unipotent group $N$ is $F$-split, hence $N(F)$ is a dense subset by Theorem~14.3.8 of~\cite{Sp}. If $\mu\in C(F)$, then the set of $F$-rational points in $\mu N$ is the $\mu$-translate of~$N(F)$, so it contains points of~$\tilde C$. Now $\tilde C$ is defined over~$F$ by Prop.~12.1.2 of~\cite{Sp}.
\end{proof}
Since the induced conjugacy class $\tilde C$ is uniquely determined by $C$ and~$P$, we shall denote it by $\Ind_P^G(C)$.

If $\sigma$ is the semisimple component of some element of a conjugacy class $C$ in~$M$, we denote by $C_\sigma$ be the set of all unipotent elements $\nu\in M_\sigma$ such that $\sigma\nu\in C$. Then $C_\sigma$ is a conjugacy class in~$M_\sigma$. The same construction can be applied to conjugacy classes in~$G$.
\begin{lemma}\label{desc}
In the situation of Proposition~\ref{defind},
\[
\Ind_P^G(C)_\sigma=\Ind_{P_\sigma}^{G_\sigma}(C_\sigma). 
\]
\end{lemma}
Note that $P_\sigma$ is a parabolic subgroup of $G_\sigma$ with Levi component $M_\sigma$ and unipotent radical~$N_\sigma$.
\begin{proof}
Let $O_\sigma=\{\delta\in C_\sigma N_\sigma\mid \sigma\delta\in\tilde C\}$. This is the intersection of an open subset of $\sigma^{-1}CN$ with $C_\sigma N_\sigma$, hence open therein. All elements of $O_\sigma$ are unipotent and conjugate under~$G_\sigma$. It remains to show that $O_\sigma$ is not empty. Take $\mu\in C$ and $n\in N$ such that $\mu n\in\tilde C$. We may assume that $\mu=\sigma\nu$ with $\nu\in C_\sigma$. By Lemma~\ref{gN}, there exist $n_1\in N$ and $n_2\in N_\sigma$ such that $\mu n=n_1\mu n_2n_1^{-1}$, hence $\mu n_2\in\tilde C$ and $\nu n_2\in O_\sigma$.
\end{proof}

\begin{prop}\label{assocpar}
Let $G$ be a connected reductive group.
\begin{enumerate}
\item[(i)] If $P_1$ and $P_2$ are parabolic subgroups of $G$ which have a Levi component $M$ in common, then $\Ind_P^G(C)=\Ind_{P'}^G(C)$ for any $M$-conjugacy class $C$ in~$M$. Thus, we can denote the induced class by $\Ind_M^G(C)$.
\item[(ii)] Let $M\subset M'$ be Levi subgroups of $G$. If $C$ is a conjugacy class of~$M$, then
\[
\Ind_M^G(C)=\Ind_{M'}^G(\Ind_M^{M'}(C)).
\]
\end{enumerate}
\end{prop}

In the proof we use the generalised flag variety~$\Pf$, whose points are the parabolic subgroups of $G$ conjugate to a given one. It is a homogeneous space for $G$, isomorphic to $G/P$ for any~$P\in\Pf$. Moreover, we have the variety
\[
\X=\{(P,g)\in\Pf\times G\mid g\in P\},
\]
on which $G$ acts by conjugation as well, with equivariant projections to $\Pf$ and to~$G$. The preimage of a $G$-invariant subset $S$ of $G$ is denoted by~$\X_S$, which generalises the Grothendieck-Springer resolution of the unipotent variety.

\begin{proof}[Proof of Proposition~\ref{assocpar}.] If $R$ is a $P$-invariant closed subset of~$P$, then the set $\{(g,h)\in G\times G\mid g^{-1}hg\in R\}$ is closed and invariant under right $P$-translations in the first component, hence the preimage of a closed subset of~$\Pf\times G$, which is easily seen to be equal to $\X_S$ for $S=[R]_G$. Its image $S$ under the projection to $G$ is closed because $\Pf$ is complete~\cite[6.2]{Sp}.

If $A$ is the centre of~$M$ and $A\reg$ the subset of those elements for which $G_a=M$, the preceding remark shows that, for any conjugacy class $C$ of~$M$, the set $[A\bar CN]_G$ is closed. It contains the dense subset $Z=[A\reg CN]_G$, which by Proposition~\ref{defind}(ii) equals~$[A\reg C]_G$. It follows that $[A\bar CN]_G$ is the closure of $[AC]_G$. Similarly, we see that $[\bar CN]_G$ is closed, hence equal to the closure of~$\tilde C$, because the latter has to contain the closure of $\tilde C\cap CN$.

Now we suppose that $C$ is unipotent. Let $Y$ be the unipotent subvariety of~$G$. Then $Y\cap A\bar CN=\bar CN$ and hence $[\bar CN]_G=Y\cap[A\bar CN]_G=Y\cap\overline{[AC]_G}$. This shows that the closure of~$\tilde C$, and hence $\tilde C$ itself, is independent of~$P$.

Next we consider a general conjugacy class~$C$. Let $P$ and $P'$ be as in~(i). By Lemma~\ref{desc}, we have $\sigma\Ind_{P_\sigma}^{G_\sigma}(C_\sigma)\subset\Ind_P^G(C)$ and $\sigma\Ind_{P'_\sigma}^{G_\sigma}(C_\sigma)\subset\Ind_{P'}^G(C)$. We already know that $\Ind_{P_\sigma}^{G_\sigma}(C_\sigma)=\Ind_{P'_\sigma}^{G_\sigma}(C_\sigma)$. This shows that $\Ind_P^G(C)$ and $\Ind_{P'}^G(C)$ have non-empty intersection, hence they must coincide.

Assertion~(ii) is easily proved as in~\cite[1.7]{LS}.
\end{proof}

\begin{theorem}\label{}
In the situation of Proposition~\ref{defind}(i), the following is true.
\begin{enumerate}
\item[(i)] $\codim_G\tilde C=\codim_M C$.
\item[(ii)] For every $\gamma\in\tilde C$ and $\mu\in C$ we have $\dim G_\gamma=\dim M_\mu$.
\item[(iii)] The set $\tilde C\cap CN$ is a $P$-conjugacy class.
\item[(iv)] For every $\delta\in\tilde C\cap CN$ we have $G_\delta\subset P$.
\end{enumerate}
\end{theorem}
In the proof, we use the natural projection $\chi$ from $G$ to the adjoint quotient $\Int(G)\backslash\!\backslash G$ (see~\cite[ch.~3]{Hu}), which generalises the Steinberg map. Two elements have the same image under $\chi$ if and only if their semisimple components are conjugate.

\begin{proof}
Properties (i) and~(ii) are equivalent, because $\codim_MC=\dim M_\mu$ and $\codim_G\tilde C=\dim G_\gamma$ for $\mu\in C$ and $\gamma\in\tilde C$. If $O$ is the $P$-orbit of an element $\delta\in CN$, then $\codim_M C+\codim_{CN}O=\codim_PO=\dim P_\delta$. If, moreover, $\delta\in\tilde C$, then
\[
\codim_G\tilde C=\codim_MC+\codim_{CN}O+\codim_{G_\delta}P_\delta.
\]
Therefore, property (i) implies that $O$ is open in~$CN$ and $P_\delta$ is open in~$G_\delta$, which is tantamount to (iii) and~(iv).

It remains to prove statement~(i). As we have seen in the proof of Proposition~\ref{assocpar}(i), the closure of $Z=[A\reg C]_G$ in $G$ is the irreducible variety $[A\bar CN]_G$. Property~(ii) is trivially satisfied for $C$ replaced by~$aC$ with $a\in A\reg$, hence (i) is also true for these~$a$, i.~e.,
\[
\dim\Ind_P^G(aC)=\dim C+\dim G-\dim M.
\]
Since the union of orbits of maximal dimension in $\bar Z$ is an open subset, it meets~$Z$, and therefore
\[
\dim\tilde C\le\dim C+\dim G-\dim M.
\]

Now assume that $C$ is unipotent. We consider the restriction $\chi_Z$ of the adjoint quotient map $\chi$ to the $G$-invariant subvariety~$\bar Z$ and its fibres $Y_a=\chi_Z^{-1}(\chi(aC))$ for $a\in A$. We know from the proof of Proposition~\ref{assocpar}(i) that $Y_1$ is the closure of~$\tilde C$. For $a\in A\reg$, we have
\[
Y_a\cap A=\{w(a)\mid w\in W\},
\]
where $W$ is the set of automorphisms of $M$ induced by the elements of the normaliser of $M$ in~$G$. The proof of Proposition~\ref{assocpar}(i) shows that, for these~$a$,
\[
Y_a=\bigcup_{w\in W}\overline{\Ind_P^G(w(a)C)}.
\]
Thus, all components of fibres contained in~$Z$ have equal dimension determined above. Since the union of fibres of minimal dimension is open in~$\bar Z$, it meets~$Z$, and therefore
\[
\dim\tilde C\ge\dim C+\dim G-\dim M.
\]

Finally, let $C$ be arbitrary. Choose $\sigma\in C_{\mathrm s}$, $\mu\in C_\sigma$ and $\gamma\in\tilde C_\sigma$. For the induction of the unipotent orbit $C_\sigma$ from $M_\sigma$ to~$G_\sigma$, assertion~(ii) is already proved. In view of $(G_\sigma)_\gamma=G_{\sigma\gamma}$ and $(M_\sigma)_\mu=M_{\sigma\mu}$, it is equivalent to the same assertion about the induction of $C$ from $M$ to~$G$.
\end{proof}

Theorem~\ref{codim}(iii) shows that the process of induction actually proceeds in two steps. We first inflate a conjugacy class $C$ of $M$ to a conjugacy class~$O$ of $P$ under the epimorphism $\pi:P\to M$ and then induce $O$ from $P$ to~$G$. We mention in passing that a similar remark applies to parabolic induction of group representations.

\begin{prop}
Given a parabolic subgroup~$P$ with unipotent radical~$N$, the subset
\[
P\infl=\{\delta\in P\mid\mathfrak n\subset(\id-\Ad(\delta))\mathfrak p\}
\]
is dense and open. It is the union of the inflations to $P$ of all conjugacy classes of a Levi component~$M$.
\end{prop} 
\begin{proof}
For every $X\in\mathfrak n$, the condition $X\in\phi_\delta'\mathfrak p$ can be expressed as the solvability of a system of linear equations, whose coefficients are regular functions of~$\delta$. Since $P\infl$ can be defined by finitely many such conditions, it is open. If we can prove the last assertion, it will follow that $P\infl$ is not empty.

Let $D$ be the conjugacy class of an arbitrary element $\delta$ in $P$, so that $\pi(D)$ is the conjugacy class $C$ of $\mu=\pi(\delta)$ in~$M$. Then we have a surjective tangential map $T_\delta(D)\to T_\mu(C)$ with kernel~$T_\delta(\mu N)$. Let $\phi_\delta:P\to D$ be the orbit map defined by $\phi_\delta(p)=p\delta p^{-1}$. An easy computation shows that the tangential map of $\phi_\delta$ at the neutral element $e$ is $\phi_\delta'=\id-\Ad(\delta)$ if we identify the tangent spaces to $P$ at $e$ and~$\delta$ with $\mathfrak p$ by right translation. The range of $\phi_\delta'$ is $T_\delta(D)$. Thus the condition $\mathfrak n\subset \phi_\delta'\mathfrak p$ is equivalent to $T_\delta(\mu N)\subset T_\delta(D)$. Together with the aforementioned surjectivity it implies that $T_\delta(D)=T_\delta(CN)$ and hence that $D$ is dense in~$CN$. The converse is clear.
\end{proof}
Given $\gamma\in G$, the action of $G$ on a generalised flag variety $\Pf$ restricts to an action of~$G_\gamma$ on the subsets $\Pf_\gamma=\{P\in\Pf\mid \gamma\in P\}$ and $\Pf_\gamma\infl=\{P\in\Pf\mid \gamma\in P\infl\}$. The former set can also be defined as the set of fixed points of~$\gamma$ in $\Pf$ and has been studied in \cite{St} and~\cite{HO}, but we are more interested in~$\Pf_\gamma\infl$. The preimage in $P$ of the union of conjugacy classes in $P/N$ which induce up to $C$ will be denoted by~$C_P$ and the action of a group on itself by inner automorphisms by~$\Int$.
\begin{prop}
For any conjugacy class $C$ of~$G$, any $\gamma\in C$ and $P\in\Pf_\gamma\infl$ we have
\[
|\Int(P/N)\backslash(C_P/N)|=|G^\gamma\backslash\Pf_\gamma\infl|.
\]
\end{prop}
\begin{proof}
The subset
\[
\X\infl=\{(P_1,\gamma_1)\in\Pf\times G\mid \gamma_1\in P_1\infl\}
\]
of $\X$ is $G$-invariant, and the natural projection $\X\to\Pf$ maps $\X\infl$ to~$\Pf\infl$. The group $G$ will then also act on the preimage $\X_C\infl$ of the conjugacy class~$C\subset G$ under the $G$-equivariant projection $\X\infl\to G$.

Any element of~$\X_C\infl$ is conjugate to an element of the form $(P_1,\gamma)$, and $(P_1,\gamma)$ is conjugate to $(P_2,\gamma)$ if and only if $P_1$ and $P_2$ are $G^\gamma$-conjugate. This yields a canonical bijection
\[
G\backslash\X_C\infl\to G^\gamma\backslash\Pf_\gamma\infl.
\]
On the other hand, any element of~$\X_C\infl$ is conjugate to an element of the form~$(P,\gamma_1)$, and $(P,\gamma_1)$ is conjugate to $(P,\gamma_2)$ if and only if $\gamma_1$ and $\gamma_2$ are $P$-conjugate. This  yields a canonical bijection
\[
G\backslash\X_C\infl\to\Int(P/N)\backslash(C_P/N).\qedhere
\]
\end{proof}
If we combine this result with Theorem~\ref{defind}(iii) and the finiteness of the number of conjugacy classes of parabolic subgroups, we obtain the following result.
\begin{cor}
For a given element $\gamma\in G$, there are only finitely many parabolic subgroups $P$ such that $\gamma\in P\infl$.
\end{cor}

\section{Prehomogeneous varieties}

Let $G$ be a connected linear algebraic group acting regularly on an irreducible variety $V$. For a rational character $\chi$ of~$G$, we denote by $\mathcal M_\chi(V)$ the space of rational functions $f$ on $V$ such that $f(gx)=\chi(g)f(x)$ for any $g\in G$ and $x\in V$. The subset of regular functions with this property will be denoted by~$\Of_\chi(V)$. We denote the group of rational characters of $G$ by~$X(G)$ and the intersection of the kernels of all rational characters of $G$ by~$G^1$.

\begin{lemma}\label{toract}
Let $D$ be a diagonalisable group acting freely and separably on a quasi-affine variety $V$. Then, for each $\omega\in X(D)$ and each $\xi\in V$, there exists a $D$-invariant open neighbourhood $U$ of~$\xi$ such that $\Of_\omega(U)\ne0$.
\end{lemma}
\begin{proof}
We consider the case of a right action. Since it is free and separable, the orbit map $D\to\xi D$ is bijective and separable. It is therefore birational, and by Zariski's main theorem, it is an isomorphism. We define a regular function $f$ on $\xi D$ by $f(\xi d)=\omega(d)$.

Let $W$ be the affine variety with coordinate algebra $\Of(W)=\Of(V)$. Then $V$ can be considered as an open subset of~$W$, and the action of $D$ extends to a regular action on~$W$. Since the $D$-orbit $\xi D$ is open in its closure, $f$ is the restriction of a regular function on an open subset of~$W$. That function may be written as the quotient of regular functions $g$ and $h$ on~$W$, so that $h|_{\xi D}f=g|_{\xi D}$.

As $D$ is diagonalisable, $\Of(W)$ is the direct sum of the spaces $\Of_\chi(W)$ over all $\chi\in X(D)$. Decomposing $g$ and $h$ accordingly, we find $\chi$ such that $h_\chi(\xi)\ne0$. As $D$ acts freely, we have $h_\chi|_{\xi D}f=g_{\chi+\omega}|_{\xi D}$, so that we may assume $h=h_\chi$, $g=g_{\chi+\omega}$. If we set $U=\{x\in V\mid h\ne0\}$, then $g/h\in\Of_\omega(U)$, $g(\xi)/h(\xi)=1$.
\end{proof}

\begin{lemma}\label{locsect}
Let $G$ be a connected linear algebraic group and $H$ a closed subgroup of~$G$. Then, for each $\omega\in X(H)$ and $g\in G$, there exists an $H$-invariant open neighbourhood $U$ of $g$ such that $\Of_\omega(U)\ne0$.
\end{lemma}
Here, we let $H$ act on $G$ from the right-hand side.
\begin{proof}
By Theorem~5.5.5 of~\cite{Sp}, the quotients $G/H$ and $G/H^1$ exist, and by~5.5.9(2) of~\cite{Sp}, the latter is quasi-affine. Moreover, the diagonalisable group $D=H/H^1$ acts freely on~$G/H^1$, the quotient being~$G/H$.

Since $G$ is connected, the natural map $G\to G/H^1$ is separable. For each $\gamma\in G$, the orbit map $H\to\gamma H$ is clearly separable, and so is therefore the copmposition $H\to G/H^1$. Since the action of $H$ on $G/H^1$ factors through that of~$D$, it follows that, for $\xi=\gamma H^1$, the orbit map $D\to\xi D$ is separable.

The natural epimorphism $H\to D$ induces an isomorphism $X(D)\to X(H)$, hence the assertion follows from Lemma~\ref{toract}.
\end{proof}

A relative invariant on a $G$-variety~$V$ is by definition a nonzero element of $\mathcal M_\chi(V)$ for some character $\chi\in X(G)$ (which is then unique). This defines a homomorphism $\eta$ from the group of relative $G$-invariants on $V$, which we denote by~$X_G(V)$, to the group $X(G)$. If $V$ is a homogeneous space, then $\mathcal M_\chi(V)=\Of_\chi(V)$. We denote the Picard group of a variety $V$ by~$\Pic(V)$.

\begin{lemma}\label{exact}
If $G$ is a connected linear algebraic group and $H$ a closed subgroup, we have an exact sequence
\[
1\longrightarrow\GL(1)\longrightarrow X_G(G/H)
\stackrel{\eta}{\longrightarrow}X(G)
\longrightarrow X(H)
\stackrel{\lambda}{\longrightarrow}\Pic(G/H).
\] 
\end{lemma}
\begin{proof}
The exactness in the terms $\GL(1)$, $X_G(G/H)$ and $X(H)$ is obvious.

Let $\pi:G\to G/H$ be the natural projection. We have, for each character $\omega\in X(H)$, the sheaf on $G/H$ which associates to every open subset $U$ the $\Of(U)$-module $\Of_\omega(\pi^{-1}U)$. By Lemma~\ref{locsect}, this sheaf is locally free over the structure sheaf of~$G/H$, and we get a homomorphism $\lambda$ from $X(H)$ to~$\Pic(G/H)$.

Now suppose that $\omega$ induces a trivial line bundle over~$G/H$. Then there is a nowhere vanishing global section $s$, which by our definition is a regular function on~$G$. According to Theorem~2.2.2 of~\cite{We}, we have $s=c\chi$ for some $c\in\GL(1)$ and $\chi\in X(G)$. The definig property of $\Of_\omega(G/H)$ implies that $\chi(gh)=\chi(g)\omega(h)$, so that $\omega$ is the restriction of a character of~$G$. The converse is clear, and so the exactness in the term $X(H)$ is proved.
\end{proof}
We say that a homogeneous space $G/H$ is special (or $G$-special if $G$ is not clear from the context) if $G$ is connected and the restriction map $X(G)\to X(H)$ is an isomorphism. In this case, $G^1$ acts transitively on~$G/H$, and as a $G^1$-homogeneous space, $G/H\cong G^1/G^1\cap H$ is special in the sense of~\cite{Ono}.

An irreducible variety $V$ with a regular $G$-action is called a prehomogeneous variety if it contains a dense (hence open) orbit $V\gen$. The elements of $V\gen$ are called generic points, and the complement $V-V\gen$ is called the singular set. If, moreover, $V$ is a vector space and the action is linear, one calls $V$ a prehomogeneous vector space. On the other hand, a prehomogeneous variety that is an affine space will be called a prehomogeneous affine space even if the affine structure is not preserved by the action. We assume henceforth that the ground field has characteristic zero, so that the orbit map $G/G^\xi\mapsto V\gen$ is an isomorphism for any $\xi\in V\gen$ (cf.~\cite{Ki}, Proposition~2.11).

\begin{lemma}\label{relinv}
For a prehomogeneous affine $G$-space~$V$, $X_G(V\gen)/\GL(1)$ is the free group generated by the defining functions of the $G$-invariant irreducible hypersurfaces in $V$.
\end{lemma}
This is proved in Theorem~2.9 of~\cite{Ki} for the case of a prehomogeneous vector space. The proof carries over to the present situation as it does not use the linearity of the action but rather the fact that $\Of(V)$ is a unique factorisation domain.

We say that a prehomogeneous affine space $V$ is special if the homogeneous space $V\gen$ is special, i.~e. if the restriction map $X(G)\to X(G^\xi)$ is an isomorphism for some (hence for every) $\xi\in V\gen$.
\begin{lemma}\label{special}
Let $V$ be a prehomogeneous affine space. Then the following conditions are equivalent.
\begin{enumerate}
\item[(i)] $V$ is special.
\item[(ii)] The singular set in $V$ has codimension larger than one.
\item[(iii)] All relative invariants of $V$ are constant.
\end{enumerate}
\end{lemma}
\begin{proof}
By Lemma~\ref{relinv}, conditions~(ii) and~(iii) are equivalent. In view of Lemma~\ref{exact}, condition~(iii) implies that the restriction map $X(G)\to X(G^\xi)$ is injective. Condition~(ii) implies that the divisor class groups of $V\gen$ coincides with that of $V$ and is therefore trivial. On the other hand, it coincides with $\Pic(V\gen)$. By Lemma~\ref{exact}, the restriction map is then surjective, and condition~(i) follows. The converse is clear.
\end{proof}
If we have a $G$-equivariant morphism of $G$-varieties $\pi:V\to W$, then any relative invariant of $W$ can be composed with $\pi$ to produce a relative invariant of~$V$. Under certain circumstances, the connection is even closer.
\begin{prop}\label{subprehom}
Let $G$ be a connected linear algebraic group and $\pi:V\to W$ an equivariant dominant morphism of irreducible $G$-varieties with irreducible fibres.
\begin{enumerate}
\item[(i)] $V$ is $G$-prehomogeneous if and only if $W$ is $G$-prehomogeneous and, for some (hence any) $\eta\in W\gen$, the fibre $\pi^{-1}(\eta)$ is $G^\eta$-prehomogeneous. In this situation, we have $W\gen=\pi(V\gen)$ and $\pi^{-1}(\eta)\gen=V\gen\cap\pi^{-1}(\eta)$.
\item[(ii)] Suppose moreover that $V$ and $W$ are affine spaces and that $\pi$ is an affine map. Then $V$ is $G$-special if and only if $W$ is $G$-special and, for some (hence any) $\eta\in W\gen$, the fibre $\pi^{-1}(\eta)$ is $G^\eta$-special.
\end{enumerate}
\end{prop}
\begin{proof} (i) If $O$ is the dense $G$-orbit in~$V$, then $\pi(O)$ is a dense $G$-orbit in~$W$, and for any $\eta\in\pi(O)$, the set $O\cap\pi^{-1}(\eta)$ is a non-empty open subset of~$\pi^{-1}(\eta)$, hence dense in that irreducible variety. Moreover, for any two points of that set, there is an element $g$ of $G$ carrying one to the other, and by the equivariance of~$\pi$, we have $g\in G^\eta$. This proves one direction of the equivalence and the characterisation of generic orbits.

Since $\pi$ is dominant, there is a non-empty open subset $U$ of $V$ such that
\[
\dim V=\dim W+\dim Z
\]
for every irreducible component $Z$ of a fibre $\pi^{-1}(\eta)$ that intersects~$U$. Now assume that $W$ is $G$-prehomogeneous and that $\pi^{-1}(\eta)$ is $G^\eta$-prehomogeneous for some $\eta\in W\gen$. Choose $\xi$ in the regular $G^\eta$-orbit of~$\pi^{-1}(\eta)$ intersected with $U$ and let $O$ be the $G$-orbit of~$\xi$. As before, we see that $O\cap\pi^{-1}(\eta)$ is the regular $G^\eta$-orbit in~$\pi^{-1}(\eta)$. An analogous dimension formula is valid for the restriction $\bar O\to\overline{\pi(O)}$ of~$\pi$ and those irreducible components of fibres which intersect some open subset of~$\bar O$. By forming differences, we get
\[
\codim_VO=\codim_W\pi(O)+\codim_{\pi^{-1}(\eta)}(O\cap\pi^{-1}(\eta)).
\]
Since both terms on the right-hand side vanish, so does the left-hand side.
\medskip

\noindent(ii) If $S$ is an irreducible subvariety of $V$, then again there is a dense open subset $U$ of $S$ such that
\begin{equation}\label{codim}
\codim_VS=\codim_W\pi(S)+\codim_{\pi^{-1}(\eta)}T
\end{equation}
for every irreducible component $T$ of a fibre $\pi^{-1}(\eta)\cap S$ that intersects~$U$. Thus, if $S$ is an invariant hypersurface in~$V$, then so is either $\pi(S)$ in $W$ or an irreducible component of $S\cap\pi^{-1}(\eta)$ in~$\pi^{-1}(\eta)$.

Conversely, assume that there is a $G$-invariant hypersurface in~$W$. Then its preimage under $\pi$ is a $G$-invariant hypersurface in~$V$. Finally assume that, for some $\eta\in W\gen$, there is a $G^\eta$-invariant irreducible hypersurface $T$ in~$\pi^{-1}(\eta)$. Then $T$ is contained in the singular set of~$V$, hence in one of its irreducible components~$S$. If $U$ is as above, then the set of elements $y\in\overline{\pi(S)}$ for which an irredicible components of $\pi^{-1}(y)\cap S$ is disjoint with $U$ is closed. If $\xi\in T$, then its $G$-orbit is contained in~$S$ and projects onto the $G$-orbit of~$\eta$, which is dense in~$W$. Thus we can replace $\eta$ and $T$ by $G$-translates so that $T\cap U\ne\emptyset$, and we conclude from~\eqref{codim} that $\codim_VS=1$.

In view of Lemma~\ref{special}, this is what had to be proved.
\end{proof}

Note that the last Proposition applies in particular to the action of a linear algebraic group on its connected unipotent normal subgroups.

\section{Canonical parabolic subgroups}

Let $G$ be a reductive linear algebraic group. From now on we assume that the ground field has characteristic zero. Then every unipotent element of $G$ is of the form $\exp X$ for a nilpotent element $X$ of the Lie algebra $\g$ of~$G$. By the Jacobson-Morozov theorem (\cite{Bou}, \S~11, Prop.~2), $X$ can be included into a Lie triple $(X,H,Y)$ in~$\g$. In case $X=0$, we have $H=Y=0$. Most of the following is well known in the special case of unipotent elements.
\begin{theorem}\label{JM}
Let $C$ be the 
conjugacy class of an element~$\gamma\in G$ with semisimple component $\sigma$ and unipotent component~$\nu$. Let $(X,H,Y)$ be a Lie triple in~$\g$ such that $\exp X=\nu$ and denote by $\g_n$ be the eigenspace of $\ad H$ in $\g$ with eigenvalue~$n$. Then the following is true.
\begin{enumerate}
\item[(i)] The subalgebras
\[
\q=\bigoplus_{n\ge0}\g_n,\qquad
\uf=\bigoplus_{n\ge2}\g_n,\qquad
\uf'=\bigoplus_{n>2}\g_n
\]
are independent of $H$ and~$Y$. The normaliser $Q$ of $\q$ in $G$ is a parabolic subgroup with Lie algebra~$\q$, called the canonical parabolic of~$\gamma$. The centraliser $L$ of $H$ is a Levi component, and $U=\exp\uf$, $U'=\exp\uf'$ are normal subgroups of~$Q$.
\item[(ii)] If $Q\can$ denotes the set of elements of~$G$ whose canonical parabolic is~$Q$, then $C\cap Q\can$ is a 
conjugacy class in~$Q$, and the set $C\cap\gamma U$ is an open and dense 
$Q_\sigma U$-orbit in~$\gamma U=\sigma U$ invariant under translations by elements of~$U'$.
\item[(iii)] The variety $V=\gamma U/U'$ is a prehomogeneous affine space under the action of~$Q_\sigma U/U'$. Left multiplication by $\gamma$ or $\sigma$ induces an isomorphism of the $Q_\sigma$-space $U_\sigma/U'_\sigma$ onto~$V^\sigma$, and the latter is a regular $L_\sigma$-prehomogeneous vector space.
\end{enumerate}
\end{theorem}
\textit{Remark.} The notion of the canonical parabolic of an element~$\gamma$ can be found in the literature in case $\gamma$ is unipotent, where the name Jacobson-Morozov parabolic is also used (cf.~\cite{Hu}, \S7.7 and \cite{CG},~\S3.8).

\textit{Remark.} In the situation of the theorem, we can choose the Lie triple inside the reductive subalgebra~$\g_\sigma$, hence the canonical parabolic of $\nu$ in~$G_\sigma$ is~$Q_\sigma$.

\textit{Remark.} In the case of unipotent~$\gamma$, assertion~(ii) amounts to the claim that $C\cap Q\can$ is an open dense 
$Q$-orbit in~$U$ invariant under translations by~$U'$.

\begin{proof}
Assertion (i) is due to Kostant~\cite{Ko} in case of unipotent~$\gamma$ and $F=\mathbb C$, but his arguments work in the general case. Since $Q$ is self-normalising, $C\cap Q\can$ is a 
$Q$-conjugacy class. Representation theory of $\mathfrak{sl}_2$ implies that the 
$Q_\sigma$-orbit of $\nu$ is open in~$U_\sigma$ and invariant under translations by~$U_\sigma'$. Moreover, the maps $U\times U_\sigma\to\gamma U$ and $U'\times U_\sigma'\to\gamma U'$ taking $(u,v)$ to $u\gamma vu^{-1}$ are surjective, whence the $Q_\sigma U$-orbit of $\gamma$ is dense in~$\gamma U$ and invariant under translations by~$U'$, while it is open on general grounds. The set $C\cap Q\can$ is clearly closed in~$C$, so its intersection with $\gamma U$ is closed in $C\cap\gamma U$. It is also dense therein, as it contains the $Q_\sigma U$-orbit of~$\gamma$, and so $C\cap\gamma U\subset Q\can$. In particular, any two elements of $C\cap\gamma U$ are conjugate by an element $q\in Q$, and since they have the same image in~$Q/U$, the element $q$ must stabilise the coset~$\sigma U$. However, the stabiliser $Q_{\sigma U}$ equals~$Q_\sigma U$, and assertion~(ii) follows.

The first part of~(iii) follows from~(ii). Suppose that $uU'\in (U/U')_\sigma$, i.e. $u\sigma u^{-1}\in\sigma U'$. From the aforementioned surjective map in case $\gamma=\sigma$ we get $u'\in U'$ and $v'\in U'_\sigma$ such that $u\sigma u^{-1}=u'\sigma v'u'^{-1}$. The left-hand side being semisimple, we have $v'=1$ and so $u^{-1}u'\in U_\sigma$. This shows that $u\in U_\sigma U'$, and so $(U/U')_\sigma=U_\sigma U'/U'\cong U_\sigma/U'_\sigma$. The generic transitivity of $L_\sigma$ follows from assertion~(ii) applied to the group~$G_\sigma$.

In order to show that the $L_\sigma$-prehomogeneous vector space  $U_\sigma/U'_\sigma$ is regular, we may assume (due to the remark preceding the proof) that $\sigma=1$, so that $U/U'\cong\g_2$ via the exponential map. Since $\ad X:\mathfrak l=\g_0\to\g_2$ is surjective, the 
$\Ad(L)$-orbit of $X$ is open in~$\g_2$. The affine space $\g_2$ is irreducible as a variety, so there can be only one open orbit~$\g_2\gen$.

Let us now fix~$H$ and hence~$L$ but let $X\in\g_2$ vary. If we fix volume forms $\mu_i$ on $\g_i$, we get a polynomial $p$ on~$\g_2$ satisfying $((\ad X)^2)^*\mu_2=p(X)\mu_{-2}$. For $X$ in~$\g_2\gen$, the map $(\ad X)^2:\g_{-2}\to\g_2$ is bijective, and hence $p(X)\ne0$. An easy calculation shows that
\[
p(\Ad(l)X)=\det(\Ad_{\g_2}(l)^2)p(X),
\]
so that $p$ is a relative invariant. For $Z\in\mathfrak l=\g_0$, the derivative of $p$ at~$X$ in the direction $[Z,X]$ equals
\[
\partial_{[Z,X]}p(X)=2\tr\ad_{\g_2}(Z)p(X).
\]
The pairing of $\g_i$ with $\g_{-i}$ defined by

\[
\langle X,X'\rangle=\tr_{\g_2}(\ad X\ad X')
\]
is nondegenerate for $i=2$, and the $L$-equivariant map $\phi:\g_2\gen\to\g_2^*\cong\g_{-2}$ defined by
\[
\langle\phi(X),X'\rangle=p(X)^{-1}\partial_{X'}p(X)
\]
satisfies
\[
\langle\phi(X),[Z,X]\rangle=2\tr\ad_{\g_2}(Z).
\]
For $X\in\g_2\gen$, there exists $Y\in\g_{-2}$ such that $(X,H,Y)$ is a Lie triple, and the $L$-orbit of~$Y$ is dense in~$\g_{-2}$. Moreover,
\[
\langle Y,[Z,X]\rangle=\langle[X,Y],Z\rangle=\langle H,Z\rangle=2\tr\ad_{\g_2}(Z)
\]
for every $Z\in\mathfrak l$, and since $\ad X:\g_0\to\g_2$ is surjective, it follows that $\phi(X)=Y$. Thus the range of $\phi$ is dense in~$\g_{-2}$, which means that the $L$-prehomogeneous vector space $\g_2$ is regular.
\end{proof}

\textit{Remark.} The prehomogeneous vector space $V^\sigma$ is of Dynkin-Kostant type in the terminology of Gyoja~\cite{Gy}, who attributes the construction of the relative invariant to Kashiwara. The proof of the last statement of the theorem could have been extracted from~\cite{Gy}, but I preferred to include the relevant short argument.

\section{A conjecture}

We are going to state a conjecture which combines the notions of inflated conjugacy class, canonical parabolic and special homogeneous space. First, we address the question how the canonical parabolic of an element varies when that element runs through an inflated conjugacy class. 

\begin{lemma}
Let $P$ be a parabolic subgroup of $G$ with unipotent radical $N$, let $C$ be a conjugacy class in~$P/N$ and $D$ its inflation to~$P$. Then among the normal subgroups $N'$ of~$P$ contained in~$N$ and such that, for each $\gamma\in D$, all elements of $D\cap\gamma N'$ have the same canonical parabolic, there is a largest one.
\end{lemma}
This subgroup will be denoted by~$N^C$. Being a unipotent algebraic group over a field of charactersitic zero, it is connected.
\begin{proof}
Let $Q$ be the canonical parabolic of an element $\gamma$ of~$D$. The condition that all elements of $D\cap\gamma N'$ have the same canonical parabolic can be stated as $D\cap\gamma N'\subset Q\can$. If $\tilde C=\Ind_P^G(C)$, this is equivalent to $D\cap\gamma N'\subset\tilde C\cap Q$. By Proposition~\ref{subprehom}(i) applied to the connected stabiliser of $\gamma N$ in~$P$, the set $D\cap\gamma N'$ is dense in~$\gamma N'$. So we see that $\gamma N'\subset Q$ and, as $\gamma\in Q$, that $N'\subset Q$. Conversely, if $N'$ is a normal subgroup of $P$ contained in~$N\cap Q$, then $D\cap\gamma N'\subset\tilde C\cap Q$.

Since any element of $D$ is of the form $\xi\gamma\xi^{-1}$ with $\xi\in P$ and since its canonical parabolic is $\xi Q\xi^{-1}$, we have to set $N^C$ equal to the intersection of the subgroups $N\cap\xi Q\xi^{-1}$.
\end{proof}

If $\gamma\in D$, then by Proposition~\ref{subprehom}(i) the variety $\gamma N/N^C$ is a prehomogeneous affine space under the action of the connected stabiliser $P_{\gamma N}$ of $\gamma N$ in~$P$. Readers who prefer to fix a Levi component $M$ of~$P$ and to identify $C$ with a conjugacy class in~$M$ may write $\gamma N=\delta N$ and $P_{\delta N}=M_\delta N$ for any $\delta\in C$.

\begin{conj}\label{norelinv}
Let $D$ be the inflation to $P$ of a conjugacy class $C$ in~$P/N$. Then, for any $\gamma\in D$, the affine space $\gamma N/N^C$ is $P_{\delta N}$-special. 
\end{conj}

The conjectured property is necessary for the applicability of Weil's generalisation of Siegel's mean-value formula (cf.~\cite{Ono}) to certain integrals on the geometric side of the trace formula. As we shall see in a sequel paper, this is a prerequisite for regrouping the contributions of the elements of $G$ to the geometric side according to their canonical parabolics.

\end{document}